\documentclass[12pt]{amsart} 
\usepackage{amsfonts,amsmath,amsthm,amssymb,latexsym,amsthm,newlfont,enumerate,color}

\newif\ifslide

\theoremstyle{plain}

\ifdefined\spacer
\newtheorem{theorem}{Theorem}
 \else
\newtheorem{theorem}{Theorem}[section]
\fi

\newtheorem{corollary}[theorem]{Corollary}
\newtheorem{lemma}[theorem]{Lemma}
\newtheorem{claim}[theorem]{Claim}

\newtheorem*{theorem*}{Theorem}

\newtheorem{proposition}[theorem]{Proposition}
\newtheorem{definition-lemma}[theorem]{Definition-Lemma}
\newtheorem{question}[theorem]{Question}
\newtheorem{red-question}[theorem]{\textcolor{red}{Question}}

\theoremstyle{definition}

\newtheorem{remark}[theorem]{Remark}

\newtheorem{example}[theorem]{Example}
\newtheorem{notation}[theorem]{Notation}

\DeclareMathOperator{\Cl}{Cl}

\renewcommand{\P}{\mathbb P}

\def\ideal#1.{I_{#1}}
\def\ring#1.{\mathcal {O}_{#1}}
\def\Spec{\operatorname {Spec}}

\def\fring#1.{\hat{\mathcal {O}}_{#1}}
\def\proj#1.{\mathbb {P}(#1)}
\def\pr #1.{\mathbb {P}^{#1}}
\def\dpr #1.{\hat{\mathbb {P}}^{#1}}
\def\af #1.{\mathbb A^{#1}}
\def\Hz #1.{\mathbb F_{#1}}
\def\Hbz #1.{\overline{\mathbb F}_{#1}}
\def\fb#1.{\underset #1 {\times}}
\def\rest#1.{\underset {\ \ring #1.} \to \otimes}
\def\au#1.{\operatorname {Aut}\,(#1)}
\def\deg#1.{\operatorname {deg } (#1)}

\def\pic#1.{\operatorname {Pic}\,(#1)}
\def\pico#1.{\operatorname{Pic}^0(#1)}
\def\picg#1.{\operatorname {Pic}^G(#1)}
\def\ner#1.{NS (#1)}
\def\rdown#1.{\llcorner#1\lrcorner}
\def\rfdown#1.{\lfloor{#1}\rfloor}
\def\rup#1.{\ulcorner{#1}\urcorner}
\def\rcup#1.{\lceil{#1}\rceil}

\def\n1#1.{\operatorname {N_1}(#1)}  
\def\cn1#1.{\overline{\operatorname {N^1}(#1)}} 
\def\cone#1.{\operatorname {NE}(#1)}     
\def\ccone#1.{\overline{\operatorname {NE}}(#1)}
\def\none#1.{\operatorname {NF}(#1)}
\def\cnone#1.{\overline{\operatorname {NF}}(#1)}
\def\mone#1.{\operatorname {NM}(#1)} 
\def\cmone#1.{\overline{\operatorname {NM}}(#1)}

\def\coef#1.{\frac{(#1-1)}{#1}}
\def\vit#1.{D_{\langle #1 \rangle}}
\def\mm#1.{\overline {M}_{0,#1}}
\def\H1#1.{H^1(#1,{\ring #1.})}
\def\ac#1.{\overline {\mathbb F}_{#1}}

\def\adj#1.{\frac {#1-1}{#1}}
\def\spn#1.{\overline{#1}}
\def\pek#1.#2.{\Cal P^{#1}(#2)}
\def\plk#1.#2.{\Cal P^{\leq #1}(#2)}
\def\ev#1.{\operatorname{ev_{#1}}}
\def\ilist#1.{{#1}_1,{#1}_2,\dots}
\def\bminv#1.{(\nu_1,s_1;\nu_2,s_2;\dots ;\nu_{#1},s_{#1};\nu_{r+1})}
\def\zinv#1.{(\nu_1,s_1;\nu_2,s_2;\dots ;\nu_{#1},s_{#1};0)}
\def\iinv#1.{(\nu_1,s_1;\nu_2,s_2;\dots ;\nu_{#1},s_{#1};\infty)}

\def\scr #1.{\mathcal #1}

\def\llist#1.#2.{{#1}_1,{#1}_2,\dots,{#1}_{#2}}
\def\ulist#1.#2.{{#1}^1,{#1}^2,\dots,{#1}^{#2}}
\def\lomitlist#1.#2.{{#1}_1,{#1}_2,\dots,\hat {{#1}_i}, \dots, {#1}_{#2}}
\def\lomitlistz#1.#2.{{#1}_0,{#1}_1,\dots,\hat {{#1}_i}, \dots, {#1}_{#2}}
\def\loc#1.#2.{\Cal O_{#1,#2}}
\def\fderiv#1.#2.{\frac {\partial #1}{\partial #2}}
\def\deriv#1.#2.{\frac {d #1}{d #2}}

\def\map#1.#2.{#1 \longrightarrow #2}
\def\rmap#1.#2.{#1 \dasharrow #2}
\def\emb#1.#2.{#1 \hookrightarrow #2}
\def\non#1.#2.{\text {Spec }#1[\epsilon]/(\epsilon)^{#2}}
\def\Hi#1.#2.{\text {Hilb}^{#1}(#2)}
\def\sym#1.#2.{\operatorname {Sym}^{#1}(#2)}
\def\Hb#1.#2.{\text {Hilb}_{#1}(#2)}
\def\Hm#1.#2.{\Hom_{#1}(#2)}
\def\prd#1.#2.{{#1}_1\cdot {#1}_2\cdots {#1}_{#2}}
\def\Bl #1.#2.{\operatorname {Bl}_{#1}#2}
\def\pl #1.#2.{#1^{\otimes #2}}
\def\mgn#1.#2.{\overline {M}_{#1,#2}}
\def\ialist#1.#2.{{#1}_1 #2 {#1}_2, #2\dots}
\def\pair#1.#2.{\langle #1, #2\rangle}
\def\vandermonde#1.#2.{\left|
\begin{matrix}
1 & 1 & 1 & \dots & 1\\
{#1}_1 & {#1}_2 & {#1}_3 & \dots & {#1}_{#2}\\
{#1}_1^2 & {#1}_2^2 & {#1}_3^2 & \dots & {#1}_{#2}^2\\
\vdots & \vdots & \vdots & \ddots & \vdots\\
{#1}_1^{#2-1} & {#1}_2^{#2-1} & {#1}_2^{#2-1} & \dots & {#1}_{#2}^{#2-1}\\
\end{matrix}
\right|
}
\def\vandermondet#1.#2.{\left|
\begin{matrix}
1 & {#1}_1   & {#1}_1^2 & \dots & {#1}_1^{#2-1}\\
1 & {#1}_2   & {#1}_2^2 & \dots & {#1}_2^{#2-1}\\
1 & {#1}_3   & {#1}_3^2 & \dots & {#1}_3^{#2-1}\\
\vdots & \vdots & \vdots & \ddots & \vdots\\
1 & {#1}_{#2}& {#1}_{#2}^2 & \dots & {#1}_{#2}^{#2-1}\\
\end{matrix}
\right|
}
\def\gr#1.#2.{\mathbb{G}(#1,#2)}

\def\alist#1.#2.#3.{{#1}_1 #2 {#1}_2 #2\dots #2 {#1}_{#3}}
\def\zlist#1.#2.#3.{#1_0 #2 #1_1 #2\dots #2 #1_{#3}}
\def\lomitlist30#1.#2.#3.{{#1}_0,{#1}_1 #2 \dots #2\hat {{#1}_i} #2\dots #2 {#1}_{#3}}
\def\lmap#1.#2.#3.{#1 \overset{#2}{\longrightarrow} #3}
\def\mes#1.#2.#3.{#1 \longrightarrow #2 \longrightarrow #3}
\def\ses#1.#2.#3.{0\longrightarrow #1 \longrightarrow #2 \longrightarrow #3 \longrightarrow 0}
\def\les#1.#2.#3.{0\longrightarrow #1 \longrightarrow #2 \longrightarrow #3}
\def\res#1.#2.#3.{#1 \longrightarrow #2 \longrightarrow #3\longrightarrow 0}
\def\Hi#1.#2.#3.{\text {Hilb}^{#1}_{#2}(#3)}
\def\ten#1.#2.#3.{#1\underset {#2}{\otimes} #3}
\def\lomitlist30#1.#2.#3.{{#1}_0 #2 {#1}_1 #2 \dots #2 \hat {{#1}_i} #2 \dots #2 {#1}_{#3}}
\def\mderiv#1.#2.#3.{\frac {d^{#3} #1}{d #2^{#3}}}

\def\Hom{\operatorname{Hom}}

\def\Proj{\operatorname{Proj}}

\def\dim{\operatorname{dim}}

\def\deg{\operatorname{deg}}

\def\Aut{\operatorname{Aut}}

\def\Sing{\operatorname{Sing}}

\def\lct{\operatorname{lct}}

\def\mult{\operatorname{mult}}

\def\rest{\operatorname{res}}

\def\C{\mathbb C}

\def\e{\Cal E}

\def\e1{E_1}
\def\e2{E_2}

\def\Q{\mathbb Q}

\def\mapdown#1{\big\downarrow\rlap{$\vcenter{\hbox{$\scriptstyle#1$}}$}}

\def\mapse#1{
{\vcenter{\hbox{$\mathop{\smash{\raise1pt\hbox{$\diagdown$}\!\lower7pt
\hbox{$\searrow$}}\vphantom{p}}\limits_{#1}\vphantom{\mapdown{}}$}}}}

\def\VR#1.{height#1pt&\omit&&\omit&&\omit&&\omit&&\omit&\cr}

\def\VRT#1.{height#1pt&\omit&&\omit&\cr}

\usepackage[colorlinks=true,pagebackref,hyperindex]{hyperref}
\usepackage{amsmath}
\usepackage{amssymb}
\usepackage{amsthm}
\usepackage[all,cmtip]{xy}

\newcommand{\bA}{\ensuremath{\mathbb{A}}}

\newcommand{\bC}{\ensuremath{\mathbb{C}}}

\newcommand{\bP}{\ensuremath{\mathbb{P}}}
\newcommand{\bQ}{\ensuremath{\mathbb{Q}}}

\newcommand{\bZ}{\ensuremath{\mathbb{Z}}}

\newcommand{\cJ}{\ensuremath{\mathcal{J}}}

\newcommand{\cL}{\ensuremath{\mathcal{L}}}

\newcommand{\cO}{\ensuremath{\mathcal{O}}}

\newcommand{\cQ}{\ensuremath{\mathcal{Q}}}

\newcommand{\cT}{\ensuremath{\mathcal{T}}}

\DeclareMathOperator{\omult}{omult}

\title{On K-stability of Fano weighted hypersurfaces}
\date{\today}

\author{Taro Sano}
\email{tarosano@math.kobe-u.ac.jp} 
\address{Department of Mathematics, Graduate School of Science, 
Kobe university, 
1-1, Rokkodai, Nada-ku, Kobe 657-8501, Japan}

\author{Luca Tasin}
\email{luca.tasin@unimi.it}
\address{Dipartimento di Matematica F.\ Enriques, Universit\`a degli Studi di Milano, Via Cesare Saldini 50, 20133 Milano, Italy}

\begin{document}



\maketitle

\begin{abstract}
Let $X \subset \mathbb{P}(a_0,\ldots,a_n)$ be a quasi-smooth weighted Fano hypersurface of degree $d$ and index $I_X$ such that $a_i |d$ for all $i$. 

If $I_X=1$, we show that, under a suitable condition, the $\alpha$-invariant of $X$ is greater than or equal to $\dim X/(\dim X+1)$ and  $X$ is K-stable. This can be applied in particular to any $X$ as above such that $\dim X \le 3$.
If $X$ is general and $I_X < \dim X$, then we show that $X$ is K-stable. 

We also give a sufficient condition for the finiteness of automorphism groups of quasi-smooth Fano weighted complete intersections. 
\end{abstract}

\section{Introduction}

Let $X$ be a Fano variety with log terminal singularities. The \textit{$\alpha$-invariant} of $X$ (also known as \textit{global log canonical threshold}) is defined as
$$
\alpha(X)=\textrm{glct}(X):= \sup \{ c \in \bQ \mid (X, cD) \text{ is log canonical for all } 0 \le D \sim_{\bQ} -K_X  \}. 
$$

This was introduced by Tian \cite{MR894378} in analytic terms to find a K\"{a}hler-Einstein metric on a Fano manifold (see also \cite[Appendix]{MR2484031}). 
This is a fundamental invariant of $X$ from many points of view. It is known that if $\alpha(X) > \dim X/(\dim X+1)$, then $X$ is K-stable (see \cite[Theorem 1.4]{MR2889147}). If $X$ is smooth then equality is enough, see \cite[Theorem 1.3]{MR3936640} (cf. \cite{Liu:2019tf}).

In \cite{Pukhlikov98, Cheltsov01} it is shown that if $X \subset \bP^{n+1}$ is a smooth Fano hypersurface of degree $n +1$, then $\alpha(X) \ge \frac{n}{n+1}$ and so $X$ is K-stable. (See also \cite{Ahmadinezhad:2020wz}, \cite{Ahmadinezhad:2021tq} for recent progress in the higher index cases.)
A natural case to then consider is that of $X \subset \bP(a_0,\ldots,a_n)$ Fano weighted hypersurface of index 1.
Del Pezzo surfaces $X \subset \bP(a_0,\ldots,a_3)$ of index 1 have been classified in \cite{Johnson-Kollar2} and the existence of K\"{a}hler-Einstein metrics has been determined (cf.\ Ibid, \cite{MR1897386}, \cite{MR1926877}, \cite{CJS10}, \cite{Liu:2020tk}). 
Fano threefolds $X \subset \bP(a_0,\ldots,a_4)$ of index 1 have been classified in \cite{Johnson-Kollar} and the K-stability of the terminal ones is well studied (cf. \cite{Cheltsov08}, \cite{MR2499678}, \cite{KOW18}, \cite{Kim:2020wz}). 
Very little is known in higher dimension except for \cite[Proposition 3.3]{Johnson-Kollar}, \cite[Theorem 1.2]{MR4056840}, \cite[Theorem 1.3]{MR4139041} to the authors' knowledge. 

The idea of this paper is to generalise the methods introduced in \cite{Pukhlikov98} and \cite{Cheltsov01,CP02} to study the $\alpha$-invariant of weighted Fano hypersurfaces in any dimension. Some consequences of our work are collected in the following result.

\begin{theorem}\label{thm:mix}
	Let $X \subset \mathbb P(a_0,\ldots, a_n)$ be a well-formed quasi-smooth weighted hypersurface of degree $d$. Assume that $a_i |d$ for all $i$ and that $X$ is Fano of index 1. Assume also that (at least) one of the following conditions holds:
	\begin{enumerate}[(i)]
		\item\label{cor:general} $X$ is general or
		\item \label{cor:ones} $a_2=\ldots=a_{n}=1$ or
		\item\label{cor:low} $\dim X \le 3$ or
		\item \label{cor:smooth} $X$ is smooth and $\dim X \le 49$.
	\end{enumerate}
Then we have 
$$
\alpha(X) \ge 
\begin{cases} \frac{d-2}{d} &\mbox{if } d=2a,\ a \ge 3 \mbox{ and (up to permutation) } \bP=\bP(a_0,\ldots,a_{n-2},2,a)  \\
\frac{d-1}{d} & \mbox{otherwise}. 
\end{cases}
$$

Moreover, $X$ is K-stable and admits a K\"ahler-Einstein metric.

If in addition $a_i \ge 2$ for any $i$ and we are not in the former case above, then $\alpha(X) \ge 1$.
\end{theorem}

\begin{remark}
Quasi-smooth Fano fourfold hypersurfaces  $X \subset \mathbb P(a_0,\ldots, a_4)$ of index 1 are classified in \cite{BK16}, see \cite{BK02} for a complete list. In this list there are 661 cases such that $a_i |d$ for any $i$.  Using the same argument as in the proof of Theorem \ref{thm:mix}\eqref{cor:low} one can check that all the members of 653 of such families satisfy the conclusion of Theorem \ref{thm:mix}. See Example \ref{ex:4-folds} for more details.
\end{remark}

An important point in this approach, which is interesting by itself, is given by the following question, whose answer is positive in the standard projective space due to \cite[Sec. 3]{Pukhlikov98} (see \cite[Statement 3.3]{Cheltsov01}).

\begin{question}\label{Q:omult}
Let $X \subset \mathbb P(a_0,\ldots, a_n)$ be a quasi-smooth weighted hypersurface of degree $d$ which is not a linear cone. 
Let $D$ be an effective divisor on $X$ such that
$
D \sim_{\mathbb Q}  H,
$
where $H:= \cO_X(1)$ and let $C$ be a curve in $X$. 
Is it true that 
$$
\omult_C D \le 1?
$$ 	
\end{question}	

Here $\omult$ is the orbifold multiplicity as in \cite[Definition 2.1.9]{Kim:2020wz} (see Remark \ref{rmk:omult}). A positive answer to Question \ref{Q:omult} implies that $(X,D)$ is log canonical outside a finite set. 
Section \ref{sec:mult} is devoted to study Question \ref{Q:omult}. In particular, in Lemma \ref{lem:smooth} we give an explicit condition  on the equation of $X$ to have positive answer to such question. In Section \ref{sec:lct} we then develop a method to compute the log canonical threshold of weighted hypersurfaces. The final consequence is the following. 

\begin{theorem}\label{thm:question} 
Let $X \subset \mathbb P(a_0,\ldots, a_n)$ be a well-formed quasi-smooth weighted hypersurface of degree $d$. Assume that $a_i |d$ for all $i$ and that $X$ is Fano of index 1. Assume that Question \ref{Q:omult} has a positive answer for $X$. 
Then we have 
$$
\alpha(X) \ge 
\begin{cases} \frac{d-2}{d} &\mbox{if } d=2a,\ a \ge 3 \mbox{ and (up to permutation) } \bP=\bP(a_0,\ldots,a_{n-2},2,a)  \\
\frac{d-1}{d} & \mbox{otherwise}. 
\end{cases}
$$

Moreover, $X$ is K-stable and admits a K\"ahler-Einstein metric.

If in addition $a_i \ge 2$ for any $i$ and we are not in the former case above, then $\alpha(X) \ge 1$.
\end{theorem}

We expect that this approach can be applied to several other cases to compute the $\alpha$-invariant of Fano weighted complete intersections besides that treated in Theorems \ref{thm:mix} and \ref{thm:question}, see for instance Example \ref{ex:a}, where the $\alpha$-invariant is computed for a hypersurface $X_{2a+1} \subset \bP(1^{(a+2)},a)$ (cf.\ \cite[Example 7.2.2]{Kim:2020wz}).
Further applications would be towards the study of birational rigidity following \cite{Pukhlikov98} and \cite{deF13,deF16} (see also \cite{SZ19}).

Finally, in Section \ref{sec:Fermat}, we study the K-stability of Fano weighted hypersurfaces of index $>1$. In Theorem \ref{thm:Fermat}, we obtain a criterion of the K-polystability (resp.\ K-semistability) of weighted hypersurfaces of Fermat type by using the argument in \cite[Corollary 4.17]{MR4309493}. 
In Theorem \ref{thm:autfinite}, we also give a sufficient condition  for the finiteness of automorphism groups of quasi-smooth weighted complete intersections. This is a generalization of \cite[Theorem 1.3]{MR4070067}. 
As a consequence, we show the following.

\begin{corollary}\label{cor:generalKstable}
Let $X=X_d \subset \bP(a_0, \ldots , a_n)$ be a well-formed quasi-smooth general Fano weighted hypersurface of degree $d$ such that the Fano index $I_X:= d- \sum_{i=0}^n a_i <  \dim X$ and $a_i | d$ for all $i$.  
Then $X$ is K-stable. 

In particular, a general smooth Fano weighted hypersurface is K-stable if it is not isomorphic to the projective space or a quadric hypersurface.  
\end{corollary}

We also exhibit some K-unstable hypersurfaces of Fermat type (Remark \ref{rmk:Kunstable}). 

The existence of a K\"{a}hler-Einstein metric on a Fano orbifold hypersurface is closely related to the existence of a Sasaki-Einstein metric on the link of the corresponding weighted homogeneous singularity (cf.\ \cite{BGK05}, \cite{MR2237105},  \cite{MR2318866}, \cite{MR3956894}). In fact, a variant of Theorem \ref{thm:Fermat} is used in \cite{LST} to construct infinitely many families of Sasaki-Einstein metrics on spheres.

\begin{notation} 
We work over the complex number field $\bC$. 

We define $\P:=\P(a_0,\ldots,a_n)$ to be the \emph{weighted projective space} with weights $a_0, \ldots, a_n$, i.e.\ $\P= \mathrm{Proj}\ \C[z_0,\ldots,z_n]$, where $z_i$ has weight $a_i$. 
For simplicity, we assume that $\bP= \P(a_0,\ldots,a_n)$ is \emph{well-formed} unless otherwise stated, i.e.\ greatest common divisor of any $n$ weights is 1 (although non-well-formed weighted projective spaces appear in the proof of Proposition \ref{P:lct}).

A closed subvariety $X=X_{d_1, \ldots, d_c} \subset \bP$ is said to be a \emph{weighted complete intersection} (WCI for short) of multidegree $(d_1,\ldots,d_c)$ if its weighted homogenous ideal in $\mathbb C[z_0,\ldots,z_n]$ is generated by a regular sequence of homogenous polynomials $\{f_j\}_{j=1}^c$ such that $\deg f_j=d_j$ for $j=1,\ldots,c$.
Let $\pi: \mathbb{A}^{n+1}\setminus \{0\} \to \P$ be the natural projection. Then $X$ is \emph{quasi-smooth} if $\pi^{-1}(X)$ is smooth.
We say that $X$ is  \emph{well-formed} if $\bP$ is well-formed and 
$
\mathrm{codim}_X (X \cap \mathrm{Sing}(\P)) \ge 2
$. 

Finally, $X_{d_1,\ldots,d_c} \subset \P$ is said to be a \emph{linear cone} if $d_j=a_i$ for some $i$ and $j$.
We recall that if $\bP$ is well-formed and $X \subset \bP$ is a WCI of dimension at least 3, then $X$ is well-formed or it is a linear cone (see \cite{Fletcher00} for generalities on WCI's).

We denote $$\mathbb{P}(\underbrace{b_1, \ldots, b_1}_{k_1}, \ldots, \underbrace{b_l, \ldots, b_l}_{k_l})$$ by $\mathbb{P}(b_1^{(k_1)}, \ldots, b_l^{(k_l)})$ for short. 

\end{notation}

\section{A multiplicity lemma}\label{sec:mult}

Let $X \subset \mathbb P(a_0,\ldots, a_n)$ be a weighted hypersurface of degree $d$ defined by a polynomial $F = F(z_0, \ldots, z_n) \in \mathbb{C}[z_0,\ldots, z_n]$. 
For $i =0,1, \ldots, n$, let 
\begin{multline*}
\pi_i \colon \bP(a_0, \ldots, 1, \ldots a_n)=: \bP_i \rightarrow \bP(a_0, \ldots, a_n); \\ 
 [x_0: \cdots : x_i : \cdots : x_n] \mapsto [x_0: \cdots : x_i^{a_i}: \cdots : x_n]
\end{multline*}
be the finite cover branched along the hyperplane $(z_i=0)$. 
Also let $\pi \colon \bP^n \rightarrow \bP(a_0, \ldots, a_n)$ defined by $[x_0: \cdots : x_n] \mapsto [x_0^{a_0}: \cdots : x_n^{a_n}]$ be the finite cover which is a composition of the morphisms $\pi_0, \ldots , \pi_n$.

Note that $Y_i:= \pi_i^{-1}(X) \subset \mathbb{P}_i$ (resp.\ $Y := \pi^{-1} (X) \subset \bP^n$) is defined by the polynomial $G_i= G_i(x_0, \ldots, x_n):= F(x_0, \ldots, x_i^{a_i}, \ldots, x_n) \in \mathbb{C}[x_0, \ldots, x_n]$ (resp. $G(x_0, \ldots , x_n):= F(x_0^{a_0}, \ldots, x_n^{a_n})$).

\begin{lemma}\label{lem:smooth}
Let $X \subset \mathbb P(a_0,\ldots, a_n)$ be a quasi-smooth weighted hypersurface of degree $d$ defined by a polynomial $F$. 
\begin{enumerate}[(i)]
	\item\label{smooth1} If $X$ is general and  $a_i |d$ for all $i$, then  $Y=\pi^{-1}(X) \subset \mathbb{P}^n$ is smooth.
	\item\label{smooth1.5} Fix $i$ such that $a_i >1$ and assume that the following condition holds: 
	 \begin{itemize}
	 	\item let $z_0^{k_0} \cdots z_n^{k_n}$ be a monomial appearing in $F$ with non-zero coefficient such that $k_i=1$. Then 
	 	$$
	 	a_i = \sum_{j \ne i : k_j >0} m_j a_j
	 	$$
	 where $m_j$ are non-negative integers.
 \end{itemize}
Then we can take an automorphism $\phi \in \Aut \bP$ such that $(\phi \circ \pi_i)^{-1}(X) \subset \bP_i$ is quasi-smooth. 

\item \label{smooth2} Assume that the following condition holds: 
	\begin{itemize}
	\item[$(\star)$] Let $z_0^{k_0} \cdots z_n^{k_n}$ be a monomial appearing in $F$ with non-zero coefficient. If $k_i=1$ for some $i$ such that $a_i >1$, then 
	$$
	a_i = \sum_{j \ne i : k_j >0} m_j a_j
	$$
	where $m_j$ are non-negative integers. 
	\end{itemize}

Then we can construct $\pi':= \pi'_1 \circ \cdots \circ \pi'_m \colon \bP^n \rightarrow \bP(a_0, \ldots, a_n)$ such that $\pi'_1, \ldots , \pi'_m$ are finite covers of the form $\phi_i \circ \pi_i$ where $\phi_i$ are automorphisms and $Y=(\pi')^{-1}(X) \subset \bP^n$ is smooth.   
\end{enumerate}
\end{lemma}	

\begin{proof}
Proof of Item \eqref{smooth1}. This follows from a direct calculation by using that
\[
\frac{\partial G}{\partial x_i} = a_i x_i^{a_i-1} \frac{\partial F}{\partial z_i}(x_0^{a_0}, \ldots, x_n^{a_n}) 
\] 
and that a general $F$ defines a quasi-smooth hypersurface in $(z_i=0 \mid i \in I) \subset \bP(a_0, \ldots, a_n)$ 
for all $I \subset \{0,1, \ldots, n \}$. 

\vspace{2mm}

\noindent Proof of Item \eqref{smooth1.5}. 
Let $Z_i:= (F_{z_0} = \cdots= \check{F_{z_i}}= \cdots= F_{z_n} =0) $, where $F_{z_j}:= \frac{\partial F}{\partial z_j}$ for $j=0,\ldots ,n$. 
Then $Z_i$ is $0$-dimensional since $X$ is quasi-smooth. 
If $Z_i \cap (z_i =0) = \emptyset$, then we see that $\pi_i^{-1}(X)=(G_i=0) \subset \bP_i$ is quasi-smooth since we compute 
\[
\frac{\partial G_i}{\partial x_i}(Q) = a_i q_i^{a_i-1} \frac{\partial F}{\partial z_i}(q_0, \ldots, q_i^{a_i}, \ldots , q_n) \neq 0  
\]
for $Q=[q_0 : \cdots : q_n] \in \pi_i^{-1}(Z_i)$. 

Hence we may assume that $Z_i \cap (z_i=0) \neq \emptyset$. 
Take $P=[p_0: \cdots : p_n] \in Z_i \cap (z_i =0)$. 
Note that $F_{z_i}(P) \neq 0$ by the quasi-smoothness of $X$. 
Since  
$$
\frac{\partial (z_0^{k_0} \cdots z_n^{k_n})}{\partial z_i} (P)
$$
can be non-zero at $P$ only if $k_i =1$ we have at least one monomial $z_0^{k_0} \cdots z_n^{k_n}$ 
appearing in $F$ such that $k_i=1$ and $p_j \neq 0$ for any $j \neq i$ such that $k_j > 0$ (otherwise we have $F_{z_i}(P) = 0$). 
Now fix such a monomial $z_0^{k_0} \cdots z_n^{k_n}$.   

Then we see that 
$$(z_i +  \lambda \prod_{j \ne i : k_j >0} z_j^{m_j} =0) \cap Z_i = \emptyset$$ 
for a general $\lambda \in \bC^*$. 
Consider now the automorphism $\phi \in \Aut \bP$ such that 
\[
\phi(z_\ell) = \begin{cases}
z_{\ell}  & ( \ell \neq i) \\
  z_i + \lambda \prod_{j \ne i : k_j >0} z_j^{m_j} & (\ell=i)
 \end{cases}. 
 \] 
Then we see that $(\phi \circ \pi_i)^{-1}(X) \subset \bP_i$ is quasi-smooth as before. 

\vspace{2mm}

\noindent Proof of Item \eqref{smooth2}. Take $i$ such that $a_i$ is a minimum among the weights bigger than 1.

If $Z_i \cap (z_i =0) = \emptyset$, then, as in the proof of Item \eqref{smooth1.5}, we have that $\pi_i^{-1}(X)=(G=0) \subset \bP_i$ is quasi-smooth and can easily check the condition $(\star)$.

If $Z_i \cap (z_i =0) \ne \emptyset$, then we can take $\phi_i \in \Aut \bP$ such that  $(\phi_i \circ \pi_i)^{-1}(X) \subset \bP_i$ is quasi-smooth as in the proof of Item \eqref{smooth1.5}. Since $a_i$ is a minimum, the condition $(\star)$ implies that $a_j=a_i$ (or $a_j =1$)  for some $j \ne i$ such that $k_j >0$ in the monomial $z_0^{k_0} \cdots z_n^{k_n}$. The latter case ($a_j =1$) is easier, so we consider the former case. Then we can take $\phi_i(z_i)= z_i + \lambda z_j$ for a general $\lambda \in \bC^*$.
We can check that the equation $G_{\phi}$ of $(\phi \circ \pi_i)^{-1}(X) \subset \bP_i$ satisfies condition ($\star$) as follows.  Note that $G_{\phi} (x_0, \ldots ,x_n) = F_{\phi}(x_0, \ldots , x_i^{a_i}, \ldots ,x_n)$, where $F_{\phi}(z_0, \ldots, z_n):= F(z_0, \ldots , \phi(z_i), \ldots, z_n)$. Also note that 
\[
F_{\phi}(z_0, \ldots, z_n) = \sum c_T z_0^{t_0} \cdots (z_i+ \lambda z_j)^{t_i} \cdots z_n^{t_n},  
\]
thus new monomials to consider are of the form $c \cdot z_0^{t_0} \cdots (z_i^{t_i-1} z_j) \cdots z_n^{t_n} $ for $j$ such that $a_j = a_i$. Hence $F_{\phi}$ satisfies ($\star$) and we check $G_{\phi}$ also satisfies ($\star$). 

Repeating this argument a finite number of times, we obtain a smooth cover $Y$ as in the statement.
\end{proof}

\begin{lemma}\label{lem:ineqa0a1a2}
	Let  $a_0, a_1, a_2$ be positive  integers such that $a_i \ne m_j a_j + m_k a_k$ for $\{i,j,k\}=\{0,1,2\}$ and non-negative integers $m_j,m_k$.  If $\gcd(a_i,a_j) =1$ for any $i \ne j$, then 
	$$
	a_0a_1a_2  - a_0 -a_1-a_2 \ge 48.
	$$
\end{lemma}	
\begin{proof}
	Write $1<a_0 < a_1 < a_2$ (equalities are not possible by assumption). Note $a_0 \ge 3$. Indeed, if $a_0=2$, then $a_1$ and $a_2$ are both odds, but then there exists a positive integer $m_0$  such that $a_2=a_1 + m_0a_0$.
	The smallest $a_0a_1a_2 -a_0 -a_1-a_2 $ is now given by $(a_0, a_1, a_2) =(3,4,5)$ and it is 48.
\end{proof}

\begin{lemma} \label{lem:special}
	Let $X \subset \mathbb P(a_0,\ldots, a_n)$ be a well-formed quasi-smooth weighted hypersurface of degree $d$. Assume that $a_i |d$ for all $i$. Assume that one of the following holds true:
	\begin{enumerate}[(i)]
		\item \label{special:2} $a_2=a_3 = \ldots=a_{n}=1$;
		\item \label{special:3} $a_3=a_4=\ldots=a_{n}=1$, $\gcd(a_i,a_j) =1$ for any $i \ne j$ and 
		$$
		d - \sum_{i : a_i >1} a_i < 48;
		$$
		\item \label{special: smooth} $\dim X\le 49$ and $X$ is a smooth Fano of index 1;
		\item \label{special:4} $X \subset \bP(a_0,a_1,a_2,1,1)$ such that $\gcd(a_0,a_1,a_2)=1$ and 	$d - \sum_{i : a_i >1} a_i < 48$;
		\item \label{special:low} $\dim X \le 3$ and $X$ is a Fano of index 1.  
		
	\end{enumerate}
Then we can construct $\pi':= \pi'_1 \circ \cdots \circ \pi'_m \colon \bP^n \rightarrow \bP(a_0, \ldots, a_n)$ such that $\pi'_1, \ldots , \pi'_m$ are finite covers of the form $\phi_i \circ \pi_i$ where $\phi_i$ are automorphisms and $Y=(\pi')^{-1}(X) \subset \bP^n$ is smooth.   
	
\end{lemma}	
\begin{proof}
	Proof of Item \eqref{special:2}. We show that condition $(\star)$ of Lemma \ref{lem:smooth}\eqref{smooth2} holds true. Assume by contradiction that there exists a monomial $M=z_0^{i_0} \cdots z_n^{i_n}$ appearing in $F$ with $i_j=1$ for some $j$ such that $a_j > 1$ and for any other $k$ such that $i_k \ge 1$, $a_k \not\vert a_j$. In particular, $a_k >1$ for any $k$ such that $i_k >0$. The monomial $M$ is thus of the form $z_jz_k^{i_k}$, with $i_k >0$. Then $d= a_j + a_ki_k$, which implies $a_k | a_j$, a contradiction. 
	
	\smallskip
	
	Proof of Item \eqref{special:3}. We can assume $a_0,a_1,a_2 >1$ by Item \eqref{special:2}. By Lemma \ref{lem:ineqa0a1a2} we get $i$ such that $a_i = m_j a_j + m_k a_k$ for $\{i,j,k\}=\{0,1,2\}$ and non-negative integers $m_j,m_k$. Hence the conditions of Lemma \ref{lem:smooth}\eqref{smooth1.5} are satisfied and we can take a cover $(\phi \circ \pi_i)^{-1}(X) \subset \bP_i$. We can now apply Item \eqref{special:2} to conclude. 

\smallskip

Proof of Item \eqref{special: smooth}. Since $X$ is smooth, the weights are pairwise coprime.  
If there are at most two weights bigger than 1, then we can apply Item \eqref{special:2}.
If there are at least three weights bigger than 1, then (using that $\dim X \le 49$) it is easy to check that the only possible cases are 
$X_{30} \subset \bP(2,3,5,1^{(21)})$ and $X_{42} \subset \bP(2,3,7,1^{(31)})$ which are covered by Item \eqref{special:3} (see Example \ref{ex:smooth} for the case $X_{60} \subset \bP(3,4,5,1^{(49)})$ of dimension 50).

\smallskip

Proof of Item \eqref{special:4}.  If $\gcd(a_i,a_j) =1$ for any $i \ne j$, then we can apply Item \eqref{special:3}, so assume that (up to reordering) $\gcd(a_0,a_1) >1$. If $a_2=1$, then we are done by Item \eqref{special:2}, so assume $a_2 >1$.  We claim that we are in the condition of applying Lemma \ref{lem:smooth}\eqref{smooth1.5} for $i=2$. In fact, consider a monomial of the form $z_0^{k_0} \cdots z_{4}^{k_4}$ such that $k_2=1$. Since $a_3=a_4=1$, the only case to check is $k_3=k_4=0$. But then $d=a_2 + k_0a_0 + k_1a_1$ and so $\gcd(a_0,a_1) \mid a_2$, a contradiction. Hence, by Lemma \ref{lem:smooth}\eqref{smooth1.5}  we can take a cover $\phi\circ \pi_2$ and then apply Item \eqref{special:2} to conclude.

\smallskip

Proof of Item \eqref{special:low}.  We will show that there is always a smooth cover as in Lemma \ref{lem:smooth}. To make the proof short, we are going to use the available classification results.

Assume first that $\dim X=2$, i.e.\ $X \subset \bP(a_0,a_1,a_2,a_3)$. By \cite[Theorem 8]{Johnson-Kollar2} there are only three possible cases satisfying $a_i | d$: $X_3 \subset \P^3$, $X_4 \subset \bP(1,1,1,2)$, $X_6 \subset \bP(1,1,2,3)$ and $X_{15} \subset \bP(3,3,5,5)$. In all cases it is immediate to see that we can apply Lemma \ref{lem:smooth}\eqref{smooth2}.

Assume now that $\dim X=3$, i.e.\ $X \subset \bP(a_0,a_1,a_2,a_3,a_4)$  (see Table \ref{t:list} for a list obtained using the classification given in \cite[Theorem 2.2]{Johnson-Kollar}).  If $a_3=a_4=1$, then $d=a_0 + a_1 + a_2 +1$, which implies $\gcd(a_0,a_1,a_2)=1$ since $a_i \mid d$ for any $i$. Then the result follows from Item \eqref{special:4}.  If  $a_0,a_1,a_2,a_3 >1$ and $a_4=1$, then we see from Table \ref{t:list} that there are only two possible cases: $X_{12} \subset \bP(2,3,3,4,1)$ and $X_{30} \subset \bP(2,3,10,15,1)$. In both cases we can apply Lemma \ref{lem:smooth}\eqref{smooth1.5} to take a cover $\phi \circ \pi_3$ and then conclude by Item \eqref{special:4}.
We are left with the case $1 < a_0 \le a_1  \le a_2 \le a_3 \le a_4$. One can check from Table \ref{t:list} that it is possible to apply Lemma \ref{lem:smooth}\eqref{smooth1.5} first with $i=4$, then after the cover with $i=3$ and finally with $i=2$. Then to conclude it is enough to use Item \eqref{special:2}.

\end{proof}

\begin{table}
\caption{Weights for Fano 3-folds of index $1$ with $a_i | d \ (\forall i)$}\label{t:list}
\begin{tabular}{ccccc|c}
$a_0$ & $a_1$ & $a_2$ & $a_3$ & $a_4$ & $d$ \\ \hline
1 & 	1 &1&  1 &1 &4 \\
1 &1 & 1 & 1 &	3 & 6 \\
1 & 1 & 1 & 2 &	2 & 6 \\
1 & 1 & 1 & 2 & 4 & 8  \\
1&	1&	1&	4&	6&	12 \\
1&	1&	2&	2&	5&	10 \\
1& 	1&	2& 	3&	6&	12 \\
1&	1&	2&	6&	9&	18\\
1&	1&	3&	4&	4&	12\\
1&	1&	3&	8&	12&	24\\
1&	1&	4&	5&	10&	20\\
1&	1&	6&	14&	21&	42\\
1&	2&	3&	3&	4&	12\\
1&	2&	3&	10&	15&	30\\
2&	2&	3&	3&	3&	12\\
2&	2&	3&	3&	9&	18\\
2&	3&	3&	14&	21&	42\\
2&	3&	5&	6&	15&	30\\
2&	4&	5&	5&	5&	20\\
2&	5&	9&	30&	45&	90\\
2&	6&	7&	7&	21&	42\\
3&	3&	3&	8&	8&	24\\
3&	3&	5&	5&	15&	30\\
3&	3&	5&	10&	10&	30\\
3&	3&	5&	20&	30&	60\\
3&	3&	15&	20&	20&	60\\
4&	4&	7&	7&	7&	28\\
5&	5&	18&	18&	45&	90\\
5&	7&	10&	14&	35&	70\\
6&	6&	11&	11&	33&	66\\
\end{tabular}
\end{table}

\begin{example}\label{ex:4-folds} In $\bP(3,4,5,4,15,30)$ consider $X$ given by
	$$	
	z_0^{17}z_1z_2 +z_0z_1^{13}z_2 +  (z_0^4+z_1^3)^5 + (z_0^5+z_2^3)^4 + (z_1^5+z_2^4)^3 -z_0^{20}-z_1^{15}-z_2^{12} + G(z_3,z_4,z_5)=0.
	$$	
	where $G$ is general of degree 60.
	Then $X$ is a quasi-smooth Fano fourfold of index 1. (The quasi-smoothness is checked by computer.) Moreover, $X \cap ({z_i =0})$ is not quasi-smooth for $i=0,1,2$ and it is not possible to perform a procedure as in Lemma \ref{lem:smooth} to get a smooth cover $Y$.
	
	Similar examples can be constructed for $X_{105} \subset \bP(40,40,30,5,3,3)$,  $X_{140} \subset \bP(70,35,20,7,5,4)$, $X_{210} \subset \bP(105,42,35,14,10,5)$, $X_{420} \subset \bP(210,140,35,28,5,3)$, $X_{714} \subset \bP(357,238,51,34,21,14)$, $X_{1386} \subset \bP(693,462,198,14,11,9)$ and $X_{1890} \subset \bP(945,630,270,27,14,5)$.  Using the classification given in  \cite{BK16},  we could check  that for any other Fano fourfold quasi-smooth hypersurface $X_d \subset \bP(a_0,\ldots,a_4)$ of index 1 such that $a_i | d$, there exists a smooth cover.  
\end{example}

\begin{example}\label{ex:smooth}  In $\bP(1^{(49)},3,4,5)$ consider $X$ given by
	\begin{multline*}
	z_0^{60} + \ldots + z_{n-3}^{60} + z_{n-2}z_{n-1}^{13}z_n +z_{n-2}^{2}z_{n-1}z_n^{10} \\ 
	+ (z_{n-2}^{4} + z_{n-1}^{3})^5  
	+ (z_{n-2}^{5} + z_{n}^{3})^4 + (z_{n-1}^{5} + z_{n}^{4})^3 -z_{n-2}^{20} - z_{n-1}^{15} - z_n^{12}=0
	\end{multline*}
	
	Then $X$ is a smooth Fano of index 1 and dimension 50. Moreover, $X \cap ({z_i =0})$ is not quasi-smooth for $i=n-2,n-1,n$ and it is not possible to perform a procedure as in Lemma \ref{lem:smooth} to get a smooth cover $Y$.
\end{example}

\begin{proposition}\label{prop:multlemma}
Let $X \subset \mathbb P(a_0,\ldots, a_n)$ be a weighted hypersurface of degree $d$. With the above notation, assume that we have a finite cover $\pi' \colon \bP^n \rightarrow \bP(a_0, \ldots, a_n)$ with the ramification formula 
\begin{equation}\label{eq:ramifformula}
K_{\bP^n} = (\pi')^* (K_{\bP} + \sum_{i=0}^n \frac{a_i -1}{a_i} H'_{ i})
\end{equation} for some hyperplanes $H'_i \in |\cO_{\bP}(a_i)|$ for $i=0, \ldots, n$ such that $Y:=(\pi')^{-1}(X) \subset \mathbb{P}^n$  is smooth and $(\pi')^* \cO_{\bP}(1) = \cO_{\bP^n}(1)$ (Such a cover exists for $X$ as in Lemmas \ref{lem:smooth} and \ref{lem:special}). Let $D$ be an effective $\bQ$-divisor on $X$ such that
$$
D \sim_{\mathbb Q}  H,
$$
where $H:= \cO_X(1)$ is the hyperplane section.

Then $(X,D)$ is log canonical outside a finite set $Z \subset X$.  
\end{proposition}

\begin{proof}
We can write $D = \frac{1}{r} D_r$ for some $r \in \mathbb{Z}_{>0}$ and $D_r \in |\mathcal{O}_{X}(r)|$. 
Since we have the ramification formula (\ref{eq:ramifformula}), we obtain 
\[
K_Y = \pi^*(K_X + \sum_{i=0}^n \frac{a_i-1}{a_i}H_i),  
\]
where $H_i:= H'_i \cap X$ for $i=0, \ldots ,n$. 
Let $\tilde{D}:= \pi^*(D)$ so that $\tilde{D} \sim_{\bQ} \cO_Y(1)$. 
Take an irreducible curve $\tilde{C} \subset Y$.  
By \cite[3.3]{Cheltsov01}, we see that 
\[
\mult_{\tilde{C}}(\tilde{D}) \le 1. 
\]
This implies that $(Y, \tilde{D})$ is log canonical on $Y \setminus \tilde{Z}$ for some finite set $\tilde{Z} \subset Y$.

Let $R:= \sum_{i=0}^n \frac{a_i-1}{a_i}H_i$. Then we have the ramification formula 
\[
K_Y + \tilde{D} = \pi^* (K_X + D + R).
\] 
By this, we see that $(X, D+R)$ is log canonical on $X \setminus Z$ 
for $Z:= \pi(\tilde{Z})$ (cf. \cite[Proposition 5.20]{KM98}), thus $(X, D)$ is log canonical on $X \setminus Z$. 
\end{proof}

\begin{remark}\label{rmk:omult}
In Proposition \ref{prop:multlemma}, we can show $\mathrm{omult}_C(D) \le 1$ for any irreducible curve $C \subset X$. We recall the definition of orbifold multiplicity (see \cite[Definition 2.1.9]{Kim:2020wz}): if $p \in X$ is a cyclic quotient singularity and $D$ is an effective divisor on $X$, then $\mathrm{omult}_p D = \mult_{p'}\pi^*D$ where $\pi: X' \to X$ is the quotient map and $p'$ is a preimage of $p$.

Let $p \in C$ be a general point and $U_p \subset X$ be a small neighborhood of $p$. 
Let $V_p:= \pi^{-1}(U_p)$ and let $\nu_p \colon \tilde{U}_p \rightarrow U_p$ be a finite cover from some smooth variety $\tilde{U}_p$  such that $\nu_p$ is \'{e}tale in codimension 1 and $U_p \simeq \tilde{U}_p / \mathbb{Z}_m$ for some $m$. Let $\tilde{V}_p$ be the normalization of the fiber product $V_p \times_{U_p} \tilde{U}_p$. 
Then we have the following diagram
\[
\xymatrix{
\tilde{V}_p \ar[r]^{\tilde{\nu}_p} \ar[d]^{\tilde{\pi}_p} & V_p \ar[d]^{\pi_p} \\
\tilde{U}_p \ar[r]^{\nu_p} & U_p
}. 
\] 
Note that $\tilde{V}_p$ is smooth and $\tilde{\nu}_p$ is \'{e}tale by the purity of branch locus. 
For an irreducible curve $\tilde{C} \subset Y$ such that $\pi(\tilde{C}) = C$, we see that $\mult_{\tilde{C}} \tilde{D} \le 1$ as above. 
This implies that $\mult_{q} (\tilde{D}) \le 1$ for $q \in \pi_p^{-1}(p)$, thus we see that $\mult_{\tilde{q}} \tilde{\nu}_p^{-1} (\tilde{D} \cap V_p) \le 1$ for $\tilde{q} \in \tilde{\nu}_p^{-1}(q)$ since $\tilde{\nu}_p$ is \'{e}tale. 
Then we see that $\mult_{\tilde{p}} \nu_p^{-1}(U_p \cap D) \le 1$ for $\tilde{p} \in \nu_p^{-1}(p)$ by considering the local homomorphism $\tilde{\pi}_p^{\sharp} \colon \cO_{\tilde{U}_p, \tilde{p}} \rightarrow \cO_{\tilde{V}_p, \tilde{q}}$ on the stalks. This implies that $\omult_C D \le 1$. 

\end{remark}

\section{A Nadel vanishing type theorem}\label{sec:Nadel}

The following is a version of Nadel vanishing for $\mathbb Q$-Cartier integral Weil divisors (not necessary Cartier) which we are going to use to compute the $\alpha$-invariant.

\begin{lemma}\label{lem:Nadel}
	Let $(X, B)$ be a log canonical pair and $D$ a $\mathbb Q$-Cartier integral Weil divisor on $X$ such that $A=D-K_X-B$ is nef and big. Let $\cJ=\cJ((X,B); -D)$ be the multiplier ideal sheaf associated to $-D$ with respect to $(X,B)$.  Then
	\begin{enumerate}
		\item \label{lem:Nadel1} there is an inclusion $\cJ \hookrightarrow \cO_X(D)$;
		\item \label{lem:Nadel2} $H^i(X, \cJ)=0$ for any $i >0$.
		\item \label{lem:Nadel3} Let $x\in X$ such that $\cO_X(D)_{x} \cong \cO_{X,x}$, i.e.\ $D$ is Cartier at $x$. Then $\cJ((X,B); -D)_x = \cJ(X,B)_x \otimes \cO_X(D)_x$, where $\cJ(X,B):= \cJ(X, B; 0)$. 
	\end{enumerate}	 	
\end{lemma}

\begin{proof}
	Let $\mu \colon W \rightarrow X$ be a log resolution of $(X,B + D)$  and define a $\bQ$-divisor $B_W$ by 
	\[
	K_W + B_W = \mu^*(K_{X} + B). 
	\]
	
	We can write $\mu^*D = \tilde{D}+\sum_{k=1}^m b_k E_k$, where $\tilde{D} \subset W$ is the strict transform of $D$ and 
	$E_1, \ldots, E_m$ are exceptional divisors of $\mu$. Since $D$ is integral, $\tilde{D}$ is a Cartier divisor on $W$.
	
	Set 
	\[
	\cL:= \cO_W(\tilde{D} + \lceil \sum b_k E_k - B_W \rceil)=\cO_W(\tilde{D} - (\lfloor -\sum b_k E_k + B_W \rfloor)). 
	\]
	Then by definition of multiplier ideal sheaf (see \cite[Definition 9.3.56]{Lazarsfeld04b})
	$$
	\cJ=\cJ((X,B);-D)= \mu_* \cL.
	$$
	
	If $E$ is the exceptional locus of $\mu$ and $Z:=\mu(E)$, we have  
	$$
	\cL_{|(W\setminus E)} \cong \cO_{X \setminus Z}(D + \lceil -B \rceil) \hookrightarrow \cO_{X \setminus Z}(D).
	$$
	Hence we have $\mu_*\cL \hookrightarrow (\mu_* \cL)^{\vee \vee} \simeq \cO_{X}(D + \lceil -B \rceil) \hookrightarrow \cO_X(D)$ and obtain an injection $\cJ \hookrightarrow \cO_X(D)$ as a composition.
	
	We now prove (2). Since 
	$$
	\tilde{D} + \sum b_k E_k - B_W  \equiv \mu^*(D) -B_W \equiv \mu^*(K_{X} + B + A) - B_W = K_W+ \mu^*(A), 
	$$ 
	the relative Kawamata-Viehweg vanishing \cite[Theorem 1-2-3]{KMM87} implies
	$$
	R^i\mu_*(\mathcal L)=0
	$$
	for $i>0$ and so by Leray spectral sequence we get that $H^i(X, \cJ)= H^i(W, \cL)$. By Kawamata-Viehweg vanishing we also get 
	$H^i(W, \cL)=0$ for $i >0$ and so (2) is proven. 
	
	Assume now that $D$ is Cartier at $x \in X$. Then 
	$$
	\cJ_x=\mu_*\cO_W(\tilde{D} + \lceil \sum b_k E_k - B_W \rceil)_x=\mu_*\cO_W(\mu^*D + \lceil - B_W \rceil)_x= \cJ(X,B)_x \otimes \cO_X(D)_x
	$$
	by the projection formula.  
\end{proof}

The following example shows that in Lemma \ref{lem:Nadel}(3) one can not simply drop the condition that $D$ is Cartier at the point $x$.

\begin{example}
	Let $X \subset \bA^3$ be the affine cone over a smooth cubic curve $C \subset \bP^2$  and let $D$ be a line through the vertex of $X$ that passes through a flex of $C$, so that $D$ is a $\bQ$-Cartier divisor ($3D$ is Cartier).  Denote by $\mu: W \to X$ the minimal resolution of $X$ with the exceptional curve $E$. Since $E^2=-3$, the following hold: 
	$$
	\mu^*D = \tilde{D} + \frac{1}{3}E    \quad \mbox{and} \quad K_W=\mu^*K_X - E.
	$$ 	
	Since $\mu_*\cO_W(\lfloor \mu^*D \rfloor)= \cO_X(\lfloor D \rfloor)= \cO_X(D)$ (use \cite[Lemma 2.11]{Nakayama04} and the fact that $D$ is integral), we get
	$$
	\cJ(X; -D)=\mu_*\cO_W(\tilde D + \lceil \frac{1}{3}E - E \rceil)= \mu_*\cO_W(\tilde D)= \mu_*\cO_W(\lfloor \mu^*D \rfloor)=\cO_X(D)
	$$
	
	On the other hand, $\cJ(X)$ is the ideal sheaf of the vertex of $X$, which is not trivial.
	We also note that the inclusion $\cJ(X;-D) \hookrightarrow \cO_X(D)$ is not strict in this case, even if the vertex is a log canonical centre of $X$. (This example reflects the necessity of the Cartier-ness of $H$ and $H_j$ at the isolated lc centre $Q$ in the proof of Proposition \ref{P:lct}.) 
\end{example}

\section{Lct computation} \label{sec:lct}

We start off with the following numerical lemma.

\begin{lemma} \label{lem:ineq} 
Let $a_0, \ldots, a_n$ and $d$ be positive integers such that: 
\begin{enumerate}
\item[(i)]\label{cond:normalized} $\gcd(a_0, \ldots, \check{a}_i, \ldots, a_n) =1$ for all $i$. 
\item[(ii)] $a_i |d$ for all $i$. 
\item[(iii)] $d = \sum_{i=0}^n a_i -1$. 
\end{enumerate}  
We use the notation $(\ast)$ for the condition
$$
(\ast)   \quad d=2a \mbox{ and (up to permutation) }  (a_0,\ldots,a_n)= (a_0,\ldots,a_{n-2},2,a) \mbox{ for some $a\ge 3$}.
$$

Set 
$$
c :=
\begin{cases} \frac{d-2}{d}=\frac{a-1}{a} &\mbox{if $a_0, \ldots ,a_n$ and $d$ satisfy ($\ast$)}\\
\frac{d-1}{d} & \mbox{otherwise}. 
\end{cases}
$$
Then, for $i,j \in \{0,1, \ldots, n \}$ such that $i \neq j$, we have 
	\begin{align}
	\label{eq:1} & -d-1+ a_i+ c d \le -1  \quad \quad \mbox{ if } a_i=1; \\
    \label{eq:ast} & -d-1 +a_i + c \frac{d}{a_i} \le -a_j  \quad  \quad  \mbox{if ($\ast$) holds and } a_i=2, a_j=a;\\
	\label{eq:j} & -d-1 +a_i + \frac{d}{a_i} \le -a_j  \quad  \quad \mbox{for all $j$ if otherwise } (a_i >1). 
    \end{align}
	We also have $c \ge \frac{n-1}{n}$ in either case. 
	The equality holds only if $(d, a_0, \ldots ,a_n) = (2a, 1,\ldots,1,2,a)$ for some $a\ge 3$ or $(n,1,\ldots,1)$. 
\end{lemma}

\begin{proof}

If $a_i=1$, then
$$
-d-1+ a_i+ c \frac{d}{a_i} \le  -d + \frac{d-1}{d}d=-1,
$$
where we used that if $(\ast)$ holds, then $\frac{d-2}{d} \le \frac{d-1}{d}$.
		
Assume that $a_i > 1$. 
Fix $j \in \{0,\ldots,n\}$. 
	
\textbf{Case 1: $d < a_ia_j$.}  Set $a_M := \max\{a_i,a_j\}$. If $ d \ge 3a_M$, then
$$
d+1 - \frac{d}{a_i} -a_i -a_j\ge d+1 -3a_M \ge 1.
$$
If $d < 3a_M$, then we must have $d=2a_M$ by the assumption that $a_i |d$ for all $i$. 
	
If $a_i=a_j$, then the condition $d=2a_i$ together with the assumption (iii) would imply that $(a_0,\ldots,a_n)= (1,a_i,a_i)$ up to permutation, which contradicts  the assumption (i).
	
If $a_i >a_j$, then $d=2a_i$ and
$$
d+1 - \frac{d}{a_i} -a_i -a_j \ge 2a_i+1 - 2 -a_i -(a_i-1) = 0.
$$
	
If $a_i <a_j$, then $d=2a_j$. We must have $a_i \ge 3$, because we are in the case $d <a_ia_j$. Note that $k:=2a_j/a_i \ge 3$ since it is an integer and $a_i <a_j$. Then
$$
d+1 - \frac{d}{a_i} -a_i -a_j = 2a_j +1 - k -a_i -a_j=\frac{ka_i}{2} +1 -k -a_i = (k-2)(\frac{a_i}{2} -1) -1
$$
which is non-negative unless $k=a_i=3$, which is not possible since it would give $9=2a_j$.

\textbf{Case 2: $d \ge a_ia_j$.}
If $a_j =1$, then 
$$
d+1 - \frac{d}{a_i} -a_i -a_j= d -\frac{d}{a_i}  -a_i = (\frac{d}{a_i}-1)(a_i -1) -1 \ge 0,
$$
where we used $d \ge 2a_i$.
So we may assume that $a_j >1$. 
	
Assume that $a_i \ge 3$.
	
Since we have $d \ge a_i a_j$, we get the required inequality by  
\[
d+1 - (a_i + a_j + \frac{d}{a_i}) = (\frac{d}{a_i} a_i -a_i -\frac{d}{a_i} +1) - a_j = (\frac{d}{a_i}-1)(a_i-1) - a_j 
\] 
\[
\ge 2(a_j-1) - a_j = a_j -2 \ge 0. 
\]

Now consider the case $a_i=2$ and recall that $d \ge 2a_j$ and $a_j |d$. 
If $d \ge 4a_j$, then the inequality follows as 
\[
 d-1- d/2 -a_j = d(1-1/2)-a_j-1 \ge 2 a_j -a_j-1 = a_j-1 >0.  
\]

If $d=3a_j$, then
$$
d+1 - \frac{d}{a_i} -a_i -a_j = 3a_j + 1 - \frac{3a_j}{2} -2 -a_j= \frac{a_j}{2} - 1 \ge 0.
$$

If $d=2a_j$, then we are in the case $(\ast)$, that is, $a=a_j \ge 3$, ($a_i=2$) and $c = (a-1)/a$ since $a_j=2$ implies $(a_0, \ldots, a_n) = (1,2,2)$ up to permutation as before. 
Then we obtain the required inequality as 
\[
d+1 - c \frac{d}{a_i} - a_i - a_j = 2a +1-ca -2-a= a-1-ca = a-1 - a\cdot \frac{a-1}{a} =0. 
\]
Finally we check the last statement. If we are in the case ($\ast$), then we have $2+a + \sum_{i=0}^{n-2} a_i = 2a+1$. Hence we have 
\[
a-1 = \sum_{i=0}^{n-2} a_i \ge n-1, 
\]	
thus we see that $c= \frac{a-1}{a} \ge \frac{n-1}{n}$ in the case ($\ast$) and the equality holds only when $a_0 = \cdots = a_{n-2}=1$. 
Otherwise, we have $$d= \sum_{i=0}^{n} a_i -1 \ge n, $$ thus see that $c= \frac{d-1}{d} \ge \frac{n-1}{n}$. The equality holds only if $a_0 = \cdots =a_n=1$. 
	
\end{proof}

\begin{lemma}\label{lem:P_i}
Let $X=X_d \subset \mathbb P(a_0,\ldots, a_n)=:\mathbb{P}$ be a quasi-smooth  weighted hypersurface of degree $d$, which is not a linear cone. Assume that $a_i | d$ for any $i=0, \ldots,n$. Then, up to a linear automorphism of $\bP$, we can assume that $P_i \notin X_d$ for any $i=0, \ldots,n$, where $P_i$ is the $i$-th coordinate point of $\mathbb P$.	
\end{lemma}

\begin{proof}
Assume there exists $P_i \in X$. Since $X$ is quasi-smooth there exists $j$ such that $\frac{\partial F}{ \partial z_j}(P_i) \ne 0$. This implies that there exists a monomial in $F$ of the form $z_jz_i^{c_i}$, i.e. $d=a_j+c_ia_i$, which tells us that $a_i |a_j$. 
We can then consider an automorphism of the form $z_j \mapsto z_j+\lambda z_i^{a_j/a_i}$ with $\lambda \in \bC^*$ general. Since $\lambda$ is general, we can apply the argument for any $P_i \in X$ to obtain the statement of the lemma. 	
\end{proof}

\begin{proposition}\label{P:lct}
Let $X \subset \mathbb P(a_0,\ldots, a_n)=:\mathbb{P}$ be a well-formed quasi-smooth  weighted hypersurface of degree $d$, which is not a linear cone.  Let $D \sim_{\bQ} H$ be a $\Q$-divisor on $X$, where $H:=\cO_{X}(1)$.  Assume that 
\begin{enumerate}[(a)]
	\item $X$ is Fano of index 1;
	\item $a_i \vert d$ for any $i=0, \ldots, n$; 
	\item the non-klt locus of $(X,D)$ is at most zero dimensional.
\end{enumerate}
We use the notation $(\ast)$ for the condition
$$
(\ast)   \quad d=2a \mbox{ and (up to permutation) }  (a_0,\ldots,a_n)= (a_0,\ldots,a_{n-2},2,a) \mbox{ for some $a\ge 3$}.
$$

Then 
$$
\lct(X,D) \ge 
\begin{cases} \frac{d-2}{d} =\frac{a-1}{a} &\mbox{if $(\ast)$ holds}   \\
	
	1  & \mbox{if $(\ast)$ does not hold and $a_i \ge 2$ for any $i$}\\ 
	\frac{d-1}{d} & \mbox{otherwise}. \\
\end{cases}
$$
\end{proposition}	

\begin{remark}
The second case of the inequality really occurs. 
For example, let $X_{6(m+l+1)} \subset \bP(2^{(2+3m)}, 3^{(1+2l)})$ be a general hypersurface of degree $6(m+l+1)$ for some $m,l \in \bZ_{>0}$. Then this satisfies the condition. 

\end{remark}

\begin{proof}

The $\mathbb{Q}$-divisor $D$ is of the form $D= D_m/m$ for some $m>0$ and $D_m \in |\mathcal{O}_X(m)|$. 	
Let $c$ be the log canonical threshold of $(X,D)$.  Assume by contradiction that $c< \frac{a-1}{a}$ if $(\ast)$ holds or $c<1$ if  $(\ast)$ does not hold and $a_i \ge 2$ for any $i$ or  $c < \frac{d-1}{d}$ otherwise.

By assumption (c), the log-canonical locus of $(X,cD)$ consists of a finite number of points. By Shokurov connectedness theorem (e.g.\ \cite[Theorem 2.8]{Cheltsov01}) we get that the $LCS(X,cD)$ consist of one single point $Q$. 

By Lemma \ref{lem:P_i} we can assume that $P_i \notin X_d$ for any $i=0, \ldots,n$, where $P_i$ is the $i$-th coordinate point of $\mathbb P$. Then we have a well defined finite morphism of degree $d/a_i$, 
\[
p_i \colon X_d \rightarrow \mathbb{P}_i:= \mathbb{P}(a_0, \ldots, \check{a_i}, \ldots, a_n)
\] 
induced by the $i$-th projection $\mathbb{P} \dashrightarrow \mathbb{P}_i$ on $\mathbb{P}$.  
Note that the finiteness of the projection follows from $P_i \not\in X_d$ and $\bP_i$ may not be well-formed. Let $c_{ij}:= \gcd (a_0, \ldots, \check{a_i}, \ldots, \check{a}_j, \ldots, a_n)$ for $j \neq i$ and $c_i:= \prod_{j \neq i} c_{ij}$. Then, by the operation as in \cite[5.7 Lemma]{Fletcher00} (cf. \cite[1.3.1]{Dolgachev}), we see that 
\[
\bP_i \simeq \bP( \frac{c_{i0}a_0}{c_i}, \ldots , \check{i}, \ldots,  \frac{c_{in}a_n}{c_i} ) = \bP(\bar{a}_0, \ldots, \check{i}, \ldots,  \bar{a}_n)=: \bar{\bP}_i,  
\]
where $\bar{a}_j:= \frac{c_{ij} a_j}{c_i}$ for $j \neq i$ ($\check{i}$ means that we skip the $i$-th term). The isomorphism follows by $\bP_i \simeq \Proj \bC[z_0^{c_{i0}}, \ldots, \check{i}, \ldots,  z_n^{c_{in}}]$ and dividing all weights by $c_i$.

Set
\[
B_{\mathbb{P}_i}:= c \cdot \frac{p_i(D_m)}{m}.
\] 

\begin{claim}\label{cl:projection}
\begin{enumerate}
\item[(i)] $p_i$ is \'{e}tale on $ X \setminus (\frac{\partial F}{\partial z_i} =0)\cup p_i^{-1}(\Sing \bP_i)$. 
\item[(ii)] There exists  $i \in \{0,\ldots,n\}$ such that $\frac{\partial F}{ \partial z_i}(Q) \neq 0$ and this implies that $Q_i:=p_i(Q) \in \mathbb{P}_i$ is an isolated lc center of the pair $(\mathbb{P}_i, B_{\mathbb{P}_i})$.
\end{enumerate}
\end{claim}

\begin{proof}[Proof of Claim]

(i) Let $$Q':= [q'_0 : \cdots : q'_n] \in X \setminus (\frac{\partial F}{\partial z_i} =0)\cup p_i^{-1}(\Sing \bP_i). $$ 
The fibre of $p_i$ over $Q'_i:= p_i(Q') = [q'_0 : \cdots : \check{q'}_i : \cdots : q'_n]$ is given by the zeros of the univariate polymonial $F(q'_0, \ldots, q'_{i-1},x,q'_{i+1}, \ldots, q'_n)$. The condition $\frac{\partial F}{ \partial z_i}(Q') \neq 0$ implies that $Q'$ is not a multiple root of  $F(q'_0, \ldots, q'_{i-1},x,q'_{i+1}, \ldots, q'_n)$ and so $p_i$ is unramified over $Q'_i$ since $p_i^{-1}(Q'_i)$ consists of $\deg p_i$ points. 

\noindent(ii) Since $X_d = (F=0) \subset \bP$ is quasi-smooth, there exists $i$ such that 
\[
\frac{\partial F}{ \partial z_i}(Q) \neq 0. 
\]
 Since $X$ and $\mathbb P_i$ are smooth in codimension 1 we conclude that $p_i$ is \'{e}tale in codimension 1 around $Q$ by (i),  and so $K_{X} + cD = p_i^*(K_{\bP_i} + B_{\bP_i})$ locally around $Q$.
Then, by a standard lemma about discrepancies (cf. \cite[Proposition 5.20]{KM98}), we get that the pair $(\bP_i, B_{\bP_i})$  is lc, but not klt and $p_i(Q)$ is an isolated lc centre.

\begin{remark} If $Q \in X$ is smooth and $p_i(Q)$ is a smooth point of $\bP_i$, then by the implicit function theorem, the projection $p_i$ induces a local analytic isomorphism of a neighbourhood $Q$ and that of $p_i(Q)$. Hence $B_{\bP_i}$ and $cD$ are also locally isomorphic and we obtain  Claim \ref{cl:projection} in a more direct way. 
\end{remark}	

\end{proof}

We have
$$
[D_m: p_i(D_m)] p_i(D_m) =(p_i)_*(D_m) \sim \cO_{\bar{\bP}_i}(m d/a_i c_i),
$$ 

where $\cO_{\bar{\bP}_i}(1) \in \Cl \bar{\bP}_i$ is the ample generator and   $[D_m: p_i(D_m)]$ is the degree of the map ${p_i}$ restricted to $D_m$. Note that we can calculate $(p_i)_*(D_m) \sim \cO_{\bar{\bP}_i}(md/a_i c_i)$ by taking some explicit hyperplane and the fact that the push-forward preserves linear equivalence (cf.\ \cite[pp.40]{Nakayama04}).

This implies that 
\begin{equation}\label{eq:adjoint}
K_{\bP_i}+B_{\bP_i} \equiv \cO_{\bar{\bP}_i}(\frac{1}{c_i}(- \sum_{j \neq i} c_{ij} a_j+ c d_i)),
\end{equation}
where 
$$
d_i:=\frac{d}{a_i[D_m: p_i(D_m)]}.
$$

We now distinguish two cases, depending on whether $Q_i= p_i(Q)  \in \bP_i$ is a smooth or a singular point.

\vspace{2mm}

\noindent {\bf Case 1}) Assume that $Q_i \in \bP_i$ is a smooth point.
Take a Weil divisor $H$ on $\bP_i$ whose class is $\cO_{\bar{\bP}_i}(1)$ and consider the multiplier ideal sheaf $\cJ=\cJ((\bP_i,B_{\bP_i});H)$. Set $\cQ:= \cO_{\bar{\bP}_i}(-1)/ \cJ$. 
By Lemma \ref{lem:Nadel}(\ref{lem:Nadel1}), we have an inclusion $\cJ \hookrightarrow \cO_{\bar{\bP}_i}(-1)$. Since $H$ is Cartier at the smooth point $Q_i$, by Lemma \ref{lem:Nadel}(\ref{lem:Nadel3}) we see that such inclusion is strict at $Q_i$. Thus the support of $\cQ$ contains $Q_i$ as a connected component and  $H^0(\cQ) \neq 0$. 

\begin{claim}\label{claim:ample-H}
We have \[
-H-(K_{\bP_i} +B_{\bP_i}) = \cO_{\bar{\bP}_i} (-1+ \frac{1}{c_i}( \sum_{j \neq i} c_{ij} a_j - c d_i) ) 
\] and it is ample as a $\bQ$-line bundle. 
\end{claim}

\begin{proof}[Proof of Claim]
The equality follows from (\ref{eq:adjoint}). 

If $\bP_i$ is well-formed, the ampleness follows from Lemma \ref{lem:ineq}, thus assume $c_i >1$, that is, $c_{ij} >1$ for some $j \neq i$. 
Then we have 
\[
\frac{1}{c_i}( \sum_{j \neq i} c_{ij} a_j - c d_i) > \sum_{j \neq i}  \frac{c_{ij} a_j}{ c_i} - \frac{d}{c_i a_i} = \sum_{j \neq i} \frac{c_{ij} a_j}{ c_i} - \sum_{j \neq i} \frac{a_j}{c_ia_i} \ge 1
\]
since $\frac{c_{ij} a_j}{c_i}, \frac{d}{c_ia_i} \in \bZ$. This implies the required ampleness. 
\end{proof}

Claim \ref{claim:ample-H} and Lemma \ref{lem:Nadel}(\ref{lem:Nadel2})  give a surjection 
\[
H^0(\bar{\bP}_i, \cO_{\bar{\bP}_i}(-1)) \twoheadrightarrow H^0(\cQ) \ne 0
\] 
which is a contradiction. 
  
\vspace{2mm}  
  
\noindent {\bf Case 2}) Assume now that $Q_i \in \bP_i$ is a singular point of $\bP_i$.  

We first deal with the case $a_i=1$. Write $Q=[q_0: \ldots : q_n]$. Since $Q_i \in \bP_i \simeq \bar{\bP}_i$ is singular, we have $\gcd\{\bar{a}_j : j \ne i \mbox{ and } q_j \ne 0\} >1$, thus 
\begin{equation}\label{eq:gcd}
\gcd\{a_j : j \ne i \mbox{ and } q_j \ne 0\} >1.
\end{equation}

The fact that $\frac{\partial F}{ \partial z_i}(Q) \ne 0$, implies that there exists a monomial $G=z_0^{b_0}\cdots z_n^{b_n}$ of degree $d$ that appears in $F$ with non-zero coefficient and such that $\frac{\partial G}{\partial z_i}(Q) \ne 0$. This means that if $b_j >0$ for $j \ne i$, then $q_j \ne 0$. By Equation \eqref{eq:gcd}, we get that $g:=\gcd\{a_j : j \ne i \mbox{ and } b_j >0\} >1$ and $g|d$ because $a_\ell | d$ for any $\ell$.  Hence $G$ must be divisible by $z_i^g$ since $a_i=1$. This gives  $q_i \ne 0$, since $\frac{\partial G}{ \partial z_i}(Q) \ne 0$.

By Equation \eqref{eq:gcd} we have that $q_j=0$ for any $j \ne i$ such that $a_j =1$ and so from the Euler identity
$$
0=dF(Q)= \sum_{\ell=0}^{n} a_\ell q_\ell \frac{\partial F}{ \partial z_\ell}(Q)
$$
we deduce that there exists $k$ such that $a_k > 1$ and $\frac{\partial F}{ \partial z_k}(Q) \ne 0$.  We can hence consider $p_k: X \to \bP_k$. Since $q_i \ne 0$ and $a_i=1$,  we see that $Q_k=p_k(Q)$ is a smooth point of $\bP_k$ and we are reduced to Case 1).

So we can assume $a_i >1$. Let $j \ne i$ such that the $j$-th coordinate of $Q$ is non-zero and note that $\cO_{\bar{\bP}_i}(-c_{ij}a_j/c_i) = \cO_{\bar{\bP}_i}(-\bar{a}_j)$ is invertible at $Q_i$. Take a Weil divisor $H_j$ on $\bP_i$ whose class is $\cO_{\bar{\bP}_i}(\bar{a}_j)$ and consider the multiplier ideal sheaf $\cJ=\cJ((\bP_i,B_{\bP_i});H_j)$. Set $\cQ:= \cO_{\bar{\bP}_i}(-\bar{a}_j)/ \cJ$. 

\begin{claim}\label{claim:ample-H_j}
 We have that 
\[
-H_j-(K_{\bP_i} +B_{\bP_i}) = \cO_{\bar{\bP}_i} (- \bar{a}_j + \frac{1}{c_i}( \sum_{k \neq i} c_{ik} a_k - c d_i))
\] 
and this is ample as a $\bQ$-line bundle.  
\end{claim} 

\begin{proof}[Proof of Claim]
The equality follows from (\ref{eq:adjoint}). 

If $c_i = 1$, then the required inequality is 
\[
-a_j + \sum_{k \neq i} a_k - c d_i  = d+1-a_i-a_j - c d_i >0
\] and it follows from (\ref{eq:ast}) and (\ref{eq:j}) in Lemma \ref{lem:ineq}, thus assume that $c_i >1$. Note that $(\ast)$ does not occur in this case. Then we have 
\begin{multline*}
\sum_{k \neq i} c_{ik} a_k - c d_i >  \sum_{k \neq i} c_{ik} a_k - d_i 
\ge (c_{ij}-1)a_j + \sum_{k \neq i} a_k - \frac{d}{a_i} \\ 
=  (c_{ij}-1)a_j + d+1-a_i - \frac{d}{a_i}
\stackrel{Lemma \ref{lem:ineq}(\ref{eq:j})}{\ge} (c_{ij} -1) a_j + a_j =c_{ij} a_j. 
\end{multline*}
This implies the required ampleness. 
\end{proof}

  As in Case 1), we reach a contradiction using Claim \ref{claim:ample-H_j} and Lemma  \ref{lem:Nadel}(\ref{lem:Nadel2}), (\ref{lem:Nadel3}).

\end{proof}

It is sometimes possible to use the above argument for the computation of the alpha invariants without the assumption (b) of Proposition \ref{P:lct} as follows. 

\begin{example}\cite[Example 7.2.2]{Kim:2020wz}\label{ex:a}
Consider a hypersurface $X= X_{2a+1} \subset \bP(1^{(a+2)},a)$ of degree $2a+1$  with $a \ge 2$ given by
$$
X= (y^2 x_1 + f(x_1, \ldots, x_{a+2})=0)
$$ 
where $f$ is general. 
Then the coordinate point $P_{y} =[0:\cdots : 0 :1] \in X$ is a singular point of $X$ and there is no automorphism to move it outside $X$. Also note that 
$$
\alpha(X) \le \lct_{P_y}(X,H_1)= \frac{a+1}{2a+1}, 
$$
where $H_1:= (x_1 =0 ) \subset X$. 

\begin{claim}
We have $\alpha(X) = \frac{a+1}{2a+1}$. 
\end{claim}

\begin{proof}[Proof of Claim] 
Let $D = \frac{D_m}{m} \sim_{\bQ} \cO_X(1)$ be an effective $\bQ$-divisor as in the proof of Proposition \ref{P:lct}. 
Note that we have a smooth cover $Y:= (z^{2a} x_1 + f(x_1, \ldots , x_{a+2}) =0) \subset \bP^n$ and can apply Proposition \ref{prop:multlemma}. Let $c:= \lct(X,0;D)$ be the log canonical threshold of $X$ with respect to $D$. 

Suppose that $c < \frac{a+1}{2a+1}$. We shall see a contradiction as Proposition \ref{P:lct}. 
By Proposition \ref{prop:multlemma}, the pair $(X, cD)$ has an isolated lc center $Q=[q_1: \cdots : q_{a+2}: r]$. Let $F:= y^2 x_1 + f(x_1, \ldots, x_{a+2})$ be the defining equation of $X$. Note that,  for $i=1,\ldots, a+2$, the projection $p_i \colon X \rightarrow \bP_i \simeq \bP(1^{(a+1)},a)$ is well-defined since the $i$-th coordinate point $P_i$ satisfies $P_i \not\in X$. On the other hand, by $P_y \in X$, the projection $p_{y} \colon X \dashrightarrow \bP_y \simeq \bP^{a+1}$ is not defined at $P_y$. 

\vspace{2mm}

\noindent \textbf{Case 1)} First consider the case where $Q= P_y=[0:\cdots : 0:1]$. 
Then the 1st projection $p_1 \colon X \rightarrow \bP_1\simeq \bP(1^{(a+1)}, a)$ is \'{e}tale at $Q$ by $\frac{\partial F}{\partial x_1}(Q) 
\neq 0$. 
Let $B_{\bP_1}:= 
c \cdot \frac{p_1(D_m)}{m}$ and $e_1:= [D_m: p_1(D_m)]$ be the degree of $p_1|_{D_m}$. Then we see that 
\[
B_{\bP_1} \sim_{\bQ} \cO_{\bP_1}(c \cdot \frac{2a+1}{e_1}). 
\]
Since $p_1(Q)=[0:\cdots :0 :1] \in \bP_1$ is a singularity of index $a$, we see that $H=\cO_{\bP_1}(a)$ is Cartier. 
Since we have 
\[
-H - (K_{\bP_1}+ B_{\bP_1}) \sim_{\bQ} \cO_{\bP_1}\left(-a + (2a+1)(1- \frac{c}{e_1})\right) 
\]
and $-a + (2a+1)(1- \frac{c}{e_1}) \ge -a+(2a+1)(1-c)> -a+(a) =0$, we see that $-H - (K_{\bP_1}+ B_{\bP_1})$ is ample. Then we can argue as the proof of Proposition \ref{P:lct} to obtain a contradiction. 

\vspace{2mm} 

\noindent \textbf{Case 2)} Consider the case where $Q \neq P_y$. 
If we have $\frac{\partial F}{\partial x_i}(Q) \neq 0$ for some $1 \le i \le a+2$, then we may assume that the coordinate point $P_{x_i}$ for $x_i$ satisfies $P_{x_i} \not\in X$ as in Lemma \ref{lem:P_i}. We see that the $i$-th projection $p_i \colon X \rightarrow \bP_i \simeq \bP(1^{(a+1)},a)$ is \'{e}tale at $Q$, thus we can argue as in Case 1) to obtain a contradiction. 

Hence we may assume that $\frac{\partial F}{\partial x_i}(Q) =0$ for 
$i=1, \ldots a+2$. Note that 
\[
\frac{\partial F}{\partial x_1} = y^2+ \frac{\partial f}{\partial x_1}, \ \frac{\partial F}{\partial x_2} = \frac{\partial f}{\partial x_2}, \ldots, 
\frac{\partial F}{\partial x_{a+2}} = \frac{\partial f}{\partial x_{a+2}}. 
\]
The point $Q = [q_1 : \cdots : q_{a+2}: r]$ satisfies that $r\neq 0$. Indeed, if $r=0$, then we have $\frac{\partial f}{\partial x_i} (q_1, \ldots q_{a+2}) =0$ for all $i$ and this is a contradiction since $f$ is general.  
We also see that $q_i \neq 0$ for some $i$ since  $Q \neq P_y$. 
By considering an automorphism $\varphi$ of $\bP(1^{(a+2)}, a)$ such that $\varphi(y)= y+ \lambda x_i^{a}$ for some $\lambda \in \bC^{*}$ and $\varphi(x_j) = x_j$ for $j=1, \ldots, a+2$, it is enough to consider the case where $r \neq 0$ and $\frac{\partial F}{\partial x_1}(Q) \neq 0$. 
Thus, by considering the projection $p_1$, we obtain the same contradiction as Case 1). 

By these, we obtain $c \ge \frac{a+1}{2a+1}$, thus the claim. 
\end{proof}

\end{example}

\section{Proof of Theorems \ref{thm:question} and \ref{thm:mix}}

\begin{proof}[Proof of Theorem \ref{thm:question}]
Let $X \subset \bP(a_0,\ldots,a_n)$ be as in the statement and let $D \sim_{\bQ} H$ be a $\Q$-divisor on $X$, where $H= \cO_X(1)$. 

Since $a_i | d$ for all $i$ and Question \ref{Q:omult} has a positive answer for $X$, we can apply Proposition \ref{P:lct} to conclude that 

$$
\lct(X,D) \ge 
\begin{cases} \frac{a-1}{a} &\mbox{if } d=2a,\ a \ge 3 \mbox{ and (up to permutation) } \bP=\bP(a_0,\ldots,a_{n-2},2,a)  \\
\frac{d-1}{d} & \mbox{otherwise}. 
\end{cases}
$$

We also get that $\alpha(X) \ge 1$ if $a_i \ge 2$ for any $i$ and we are not in case ($\ast$) (note that $X$ is not smooth in this case by \cite[Lemma 3.3]{MR4124110} or \cite[Theorem 1.2]{PST17}).

Assume now that ($\ast$) holds and $a_0=\ldots=a_{n-2}=1$. Then we have  $\dim X_{2a} = a-1$ since the Fano index is $1$, thus we only see the K-semistability of $X_{2a}$ from the criterion \cite[Theorem 1.4]{MR2889147}.  
Nevertheless we see the K-stability of $X_{2a}$ as follows. 
If $a$ is odd, then we see that $X_{2a}$ is smooth, thus $X_{2a}$ is K-stable by \cite[Theorem 1.3]{MR3936640}.
If $a$ is even, then $X_{2a}$ has only $1/2(1,\ldots ,1)$-singularities. 
We see that the singularities are not weakly exceptional by \cite[Corollary 3.20]{MR2860982} since a cyclic quotient singularity is defined by a reducible representation of a cyclic group.  
By this and \cite[Theorem 3.1]{Liu:2019tf}, we see that $X_{2a}$ is K-stable when $a$ is even. 

If  ($\ast$) holds but we do not have $a_0=\ldots=a_{n-2}=1$, then  $d= \sum_{i=0}^n a_i -1 > n+1=\dim X +2$.  If   ($\ast$) does not hold, then we have $d= \sum_{i=0}^n a_i -1 > n=\dim X +1$ since we may assume that $\bP(a_0, \ldots, a_n) \not\simeq \bP^n$.  In all these cases we obtain the K-stability of $X_d$ from \cite[Theorem 1.4]{MR2889147}. 

Since $X$ is K-stable, we see that $X$ admits a K\"ahler-Einstein metric  (cf. \cite[Section 6]{DK01}, \cite{Li:2019uj}, \cite{Li:2019uh}).
\end{proof}

\begin{proof}[Proof of Theorem \ref{thm:mix}]

Proof of Item \eqref{cor:general}. Theorem \ref{thm:question} can be applied to any Fermat type hypersurface $X=\{z_0^{d/a_0} + \ldots + z_n^{d/a_n}=0\} \subset \bP(a_0,\ldots,a_n)$ since it has a smooth cover as in Lemma \ref{lem:smooth}. Then, by the openness of (uniform) K-stability (\cite{BL18}, \cite{BLX19}, \cite{Liu:2021vj}), we obtain the K-stability of general $X_d$.   
Another proof can be given following the proof of Theorem \ref{thm:question}, replacing Lemma \ref{lem:smooth}\eqref{smooth2} with Lemma \ref{lem:smooth}\eqref{smooth1}.

\smallskip

Proof of Items \eqref{cor:ones}, \eqref{cor:low} and \eqref{cor:smooth}.  In all these cases, Lemma \ref{lem:special} assures that we have a smooth cover so that we can apply Proposition \ref{prop:multlemma}. The conclusion follows then by Proposition \ref{P:lct}.
\end{proof}

\section{Automorphism groups and Fano weighted hypersurfaces of Fermat type}\label{sec:Fermat}

Adapting the argument in \cite[Corollary 4.17]{MR4309493} using the criterion \cite[Corollary 1.6]{Fujita:2017wm}, we have the following criterion for the K-polystability of  Fano weighted hypersurfaces of Fermat type. 

\begin{theorem}\label{thm:Fermat}
Let $X_d := (z_0^{d_0} + \cdots + z_n^{d_n}=0) \subset \bP(a_0, \ldots, a_n)$ be a quasi-smooth Fano hypersurface of degree $d$ such that $a_i |d$ for all $i$ and $d_i:= \frac{d}{a_i} \ge 2$. Let $I_X:= \sum_{i=0}^n a_i - d$ be the Fano index of $X_d$. Assume that $a_0 \le a_1 \le \cdots \le a_n$. 

Then $X_d$ is K-polystable (resp.\ K-semistable) if and only if $I_X < na_0$ (resp.\ $I_X \le na_0$). In particular, $X_d$ is K-polystable when $X_d$ is smooth. 
\end{theorem}

\begin{remark}
It is known that the condition $I_X \le na_0$ is a necessary condition for the existence of a K\"{a}hler-Einstein metric on a well-formed quasi-smooth hypersurface $X \subset \bP(a_0, \ldots, a_n)$ (\cite[(3.23)]{MR2318866}, \cite[Example 1.8]{CJS10}). 
\end{remark}

\begin{proof}
 We shall show this by adapting \cite[Corollary 4.17]{MR4309493}. Consider $\bP^n$ with the coordinate $[w_0 : \cdots : w_n]$. Then 
$X_d$ admits a Galois covering $\pi \colon X_d \rightarrow H \subset \bP^n$ defined by $[z_0 : \cdots : z_n] \mapsto [z_0^{d_0} : \cdots : z_n^{d_n}]$, where $H := (w_0 + \cdots + w_n =0) \subset \bP^n$. Indeed, we check that the Galois group of $\pi$ is $\bigoplus_{i=0}^n \bZ/(d_i)$ since it is induced from the injection $\bC[w_0, \ldots, w_n] \hookrightarrow \bC[z_0, \ldots, z_n]$ determined by $w_i \mapsto z_i^{d_i}$ for $i=0, \ldots n$. 
Let $H_i:= (w_i =0) \subset H \simeq \bP^{n-1}$ for $i=0, \ldots, n$. 
Then we see that $\bigcup_{i=0}^n H_i \subset H$ is an SNC divisor and 
\[
K_{X_d} = \pi^* (K_H + \sum_{i=0}^n (1-\frac{1}{d_i} ) H_i). 
\]
By \cite[Corollary 4.13]{MR4309493}, in order to check the K-polystability of $X_d$, it is enough to check the K-polystability of the log Fano hyperplane arrangement $(H, \sum_{i=0}^n (1-\frac{1}{d_i} ) H_i )$. 
By $H \simeq \bP^{n-1}$ and the criterion \cite[Corollary 1.6]{Fujita:2017wm}, the above pair is K-semistable (resp. uniformly K-stable) if and only if 
\[
k \sum_{i=0}^n (1- \frac{1}{d_i}) \ge n \sum_{j=1}^k (1- \frac{1}{d_{i_j}}) \text{ (resp. $k \sum_{i=0}^n (1- \frac{1}{d_i}) > n \sum_{j=1}^k (1- \frac{1}{d_{i_j}})$ }
\] 
for any $k=1, \ldots, n-1$ and $0\le i_1 < \cdots <i_k \le n$. 
The $(\text{L.H.S.} - \text{R.H.S.})$ is equal to  
\begin{multline*}
k (n+1 - \sum_{i=0}^n \frac{1}{d_i}) - n (k- \sum_{j=1}^k \frac{1}{d_{i_j}}) = k- k \sum_{i=0}^n \frac{1}{d_i} + n \sum_{j=1}^k \frac{1}{d_{i_j}} \\ 
= \frac{k}{d} \left( d - \sum_{i=0}^n a_i + \frac{n}{k} \sum_{j=1}^k a_{i_j}\right) = \frac{k}{d}\left(-I_X + \frac{n}{k} \sum_{j=1}^k a_{i_j}\right). 
\end{multline*}
Then the K-semistability  (resp. uniform K-stability) of the arrangement is equivalent to the non-negativity (resp. the positivity) of the term
$-I_X + \frac{n}{k} \sum_{j=1}^k a_{i_j}. 
$
Note that we have 
\[
\min \left\{ -I_X + \frac{n}{k} \sum_{j=1}^k a_{i_j} \mid k=1,\ldots, n-1,  0 \le i_1 < \cdots < i_k \le n \right\} =  - I_X +n a_0 
\] 
 since $a_0 = \min \{a_0, \ldots, a_n \}$ and 
$
 \frac{1}{k} \sum_{j=1}^k a_{i_j} \ge a_0
$. 
Hence we check that the positivity is equivalent to the positivity of $-I_X + n a_0$.

When $X_d$ is smooth, we always have the K-polystability since $X \simeq \bP^{n-1}$ if $I_X \ge n$. 
\end{proof}

\begin{remark}\label{rmk:Kunstable}

It is easy to find a quasi-smooth K-unstable hypersurface of Fermat type $X_d \subset \bP(a_0, \ldots, a_n)$ as follows. 
For example, let 
\[
X_6:= (y_1^3 + \cdots + y_m^3 + z_1^2 + \cdots + z_l^2 = 0) \subset \bP(2^{(m)}, 3^{(l)})
\] for some $m \ge 1, l \ge 5$. 
Then we see that $I_X = 2m+3l -6$ and $n=m+l-1$, $a_0 =2$, thus  we have 
\[
-I_X + n a_0 = -(2m+3l-6) + 2(m+l-1) = -l+4 <0. 
\] 
Hence we see that $X_6$ is not K-semistable by Theorem \ref{thm:Fermat}. 

One can check that $H^0(X_6, \cT_{X_6}) \neq 0$ and $|\Aut(X_6)| = \infty$. 
Indeed, since $z_1^2+z_2^2 = (z_1 + \sqrt{-1}z_2)(z_1 - \sqrt{-1}z_2)$, we obtain automorphisms as in \cite[Example 5.1]{MR4223973}. 

Let $X_{12} \subset \bP(3^{(m)}, 4^{(l)})$ for $m \ge 1, l \ge 10$. 
Then we see that $X_{12}$ is also K-unstable by Theorem \ref{thm:Fermat} and 
\[
-I_X + n a_0 = -(3m+4l -12) + 3(m+l-1) = -l +9. 
\] We also see that $\Aut(X_{12})$ is finite by Theorem \ref{thm:autfinite}(i) and $12 >4+4$. 

\end{remark}

The following proposition is useful to study the infinitesimal  automorphisms of weighted complete intersections. 

\begin{proposition}\label{prop:surjections} 
Let $n \ge 3$ and $X = X_{d_1, \ldots, d_c} \subset \bP(a_0, \ldots , a_n)$ be a quasi-smooth weighted complete intersection which is not a linear cone. Let $C_X \subset \bA^{n+1}$ be the affine cone of $X$ and $\pi \colon C'_X \rightarrow X$ be the quotient morphism from $C'_X:= C_X \setminus \{0 \}$. Let $R:= H^0(C_X, \cO_{C_X})$ be the coordinate ring of $C_X$. 

\begin{enumerate}
\item[(i)] We have a surjection $H^0(C_X, \cT_{C_X})^{\bC^*} \twoheadrightarrow H^0(X, \cT_X)$, where $\cT_{C_X}, \cT_X$ are tangent sheaves and $M^{\bC^*} \subset M$ is the $\bC^*$-invariant part of an $R$-module $M$ with a $\bC^*$-action. 
\item[(ii)] We have a surjection 
\[
\left\{ \wedge^{n-c} (\bigoplus_{i=0}^n R(a_i)) \right\}_{I_X} \rightarrow H^0(C_X, \cT_{C_X})^{\bC^*},
\]
where $R(a_i) = R$ is a free $R$-module with a $\bC^*$-action such that $\lambda \cdot s_j= \lambda^{a_i+j} s_j$ for a homogeneous $s_j \in R_j$ of degree $j$ and, for an $R$-module $M$ with a $\bC^*$-action, let $M_{I_X}:= \{x \in M \mid \lambda \cdot x = \lambda^{I_X} x \}$. 

\end{enumerate}

\end{proposition}

\begin{proof}
(i) Note that $C'_X \simeq \Spec_X \bigoplus_{i \in \bZ} \cO_X(i)$ 
and we have an exact sequence 
\[
0 \rightarrow (\cT_{C'_X/ X})^{\vee \vee} \rightarrow \cT_{C'_X} \rightarrow (\pi^*\cT_{X})^{\vee \vee} \rightarrow 0. 
\]
Note also that $(\cT_{C'_X/ X})^{\vee \vee} \simeq (\pi^* \cO_X(1))^{\vee \vee}$ by the above description of $C'_X$. 
By taking $\pi_*$ and its $\bC^*$-invariant part, we obtain an exact sequence 
\[
H^0(C'_X, \cT_{C'_X})^{\bC^*} \rightarrow H^0(X, \cT_X) \rightarrow H^1(X, \cO_X(1)). 
\] 
Since $\cT_{C_X}$ is reflexive and $H^0( C_X, \cT_{C_X}) \simeq H^0(C'_X, \cT_{C'_X})$, the required surjectivity follows from $H^1(X, \cO_X(1)) =0$ by $\dim X = n-1 \ge 2$ (cf. \cite{Dolgachev},  \cite[Lemma 7.1]{Fletcher00}).

\vspace{2mm}

\noindent(ii) Note that $\dim C_X = n-c+1$ and we have $\cT_{C_X} \simeq \cT_{C_X} \otimes \omega_{C_X} \otimes \omega_{C_X}^{-1} \simeq \Omega^{n-c}_{C_X} \otimes \omega_{C_X}^{-1}$. Since the generator $s \in H^0(\omega_{C_X}^{-1})$ satisfies that $\lambda \cdot s = \lambda^{-I_X} s$, we see that 
\[
H^0(C_X, \cT_{C_X})^{\bC^*} \simeq H^0(C_X, \Omega^{n-c}_{C_X} \otimes \omega_{C_X}^{-1})^{\bC^*} \simeq H^0(C_X, \Omega^{n-c}_{C_X})_{I_X}.  
\]
The surjection $\Omega^{n-c}_{\bA^{n+1}}|_{C_X} \twoheadrightarrow \Omega^{n-c}_{C_X}$ induces a surjection 
\[
H^0(C_X, \Omega^{n-c}_{\bA^{n+1}}|_{C_X})_{I_X} \rightarrow H^0(C_X, \Omega_{C_X}^{n-c})_{I_X}.
\] 
Since we have $H^0(C_X, \Omega^{n-c}_{\bA^{n+1}}|_{C_X})\simeq 
\wedge^{n-c}(\bigoplus_{i=0}^n R(a_i))$, we obtain the required surjection. 
\end{proof}

We give a sufficient condition for the finiteness of the automorphism groups of quasi-smooth weighted complete intersections as follows. 
This is a generalization of \cite[Theorem 1.3]{MR4070067} which is based on the calculations in \cite{MR636200}. 

\begin{theorem}\label{thm:autfinite}
Let $X= X_{d_1, \ldots, d_c} \subset \bP(a_0, \ldots, a_n)$ be a quasi-smooth Fano weighted complete intersection which is not a linear cone with $a_0 \le \cdots \le a_n$. Then we have the following. 
\begin{enumerate}
\item[(i)] The automorphism group $\Aut(X)$ is finite if 
\begin{equation*}\label{eq:finitecond}
\sum_{j=1}^c d_j > a_{n-c} + \cdots + a_n. 
\end{equation*} 
\item[(ii)] In particular, $\Aut (X)$ is finite if $I_X < \dim X$. 
\end{enumerate}
\end{theorem}

\begin{proof}
(i) To prove the first statement, it is enough to show $H^0(X, \cT_X) =0$ when  $a_0+ \cdots + a_{n-c-1} > I_X$. 
Let $L:=  \wedge^{n-c} (\bigoplus_{i=0}^n R(a_i))$ and $L= \bigoplus_{i \in \bZ} L_i$ be the eigen-decomposition with respect to the $\bC^*$-action on $L$. We check that $L_i =0$ if $i < a_0+ \cdots a_{n-c-1}$ by the construction of $L$. In particular, we obtain $L_{I_X} =0$, thus the first statement follows by Proposition \ref{prop:surjections}. 

\vspace{2mm}

\noindent(ii) This follows from $a_0 + \cdots a_{n-c-1} \ge n-c = \dim X$.

\end{proof}

\begin{proof}[Proof of Corollary \ref{cor:generalKstable}]
If $X_d$ is a Fermat hypersurface, then $X_d$ is K-stable since $X_d$ is K-polystable by  Theorem \ref{thm:Fermat} and its automorphism group is finite by Theorem \ref{thm:autfinite}. This implies that general hypersurfaces are K-stable by the openness of uniform K-stability (\cite[Theorem A]{BL18}, \cite{BLX19}) and the equivalence of uniform K-stability and K-stability \cite[Theorem 1.5]{Liu:2021vj}. 
\end{proof}

\section*{Acknowledgement}
We thank Giulio Codogni, Kento Fujita, Yuchen Liu, Yuji Odaka, Takuzo Okada and Andrea Petracci for valuable comments. 
We are grateful to the referee for helpful comments. 
Part of this project has been realized when the first author visited Universit\`a degli Studi di Milano. He would like to thank the university for the support. 
The first author was partially supported by JSPS KAKENHI Grant Numbers JP17H06127, JP19K14509. 
The second author is member of the GNSAGA group of INdAM.

\bibliographystyle{amsalpha}
\bibliography{Library}

\end{document}